\documentclass[preprint,authoryear,12pt]{elsarticle}

\usepackage{amssymb}
\usepackage{amsfonts}
\usepackage{amsmath}

\makeatletter
\def\ps@pprintTitle{
 \let\@oddhead\@empty
 \let\@evenhead\@empty
 \def\@oddfoot{\centerline{\thepage}}
 \let\@evenfoot\@oddfoot}
\makeatother

\journal{Statistics and Probability Letters}
\newtheorem{theorem}{Theorem}

\newtheorem{corollary}[theorem]{Corollary}

\newtheorem{lemma}[theorem]{Lemma}

\newtheorem{remark}[theorem]{Remark}

\newenvironment{proof}[1][Proof]{\noindent\textbf{#1.} }{\ \rule{0.5em}{0.5em}}
\begin{document}

\begin{frontmatter}%
\title{On the Log Quantile Difference of the Temporal Aggregation of a Stable Moving Average Process}%
\author{A. W. Barker}%
\address{Department of Statistics, Macquarie University, Sydney, NSW 2109 Australia }%

\begin{abstract}%
A formula is derived for the log quantile difference of the temporal
aggregation of some types of stable moving average processes, MA(q). The
shape of the log quantile difference as a function of the aggregation level
is examined and shown to be dependent on the parameters of the moving
average process but not the quantile levels. The classes of invertible,
stable MA(1)\ and MA(2)\ processes are examined in more detail.
\end{abstract}%
\end{frontmatter}%

\section{Introduction}

We recall some basic facts and definitions about stable moving average
processes, temporal aggregation and log quantile differences.

\subsection{Stable Moving Average Processes}

Let $\left\{ X_{t}\right\} $ be the moving average process of order $q,$%
\begin{equation}
X_{t}=\sum_{j=0}^{q}\theta _{j}e_{t-j}  \label{eq:INT:1}
\end{equation}%
where $\theta _{0}=1$ and $\left\{ e_{t}\right\} $ is an independently and
identically distributed (iid)\ sequence of stable random variables such that%
\begin{equation}
e_{t}\sim S_{\alpha }^{0}\left( \beta ^{\left( 0\right) },\gamma ^{\left(
0\right) },\delta ^{\left( 0\right) }\right)   \label{eq:INT:2}
\end{equation}%
using the $S^{0}$ parameterisation of stable distributions in \cite{ci:N98}.
Let $\theta $ denote the $q+1$ dimensional vector of moving average
parameters%
\begin{equation}
\theta =\left( \theta _{0},\ldots ,\theta _{q}\right) ^{\prime }.
\label{eq:INT:3}
\end{equation}%
In $\left( \ref{eq:INT:2}\right) $ and the remainder of this paper, with the
addition of various subscripts or superscripts, we use $\alpha ,\beta
,\gamma $ and $\delta $ to denote respectively the stability, skewness,
scale and location parameters of a stable distribution. The $S^{0}$
parameterisation has the following useful properties.

\begin{description}
\item[(P1)] If $Y\sim S_{\alpha }^{0}\left( \beta ,\gamma ,\delta \right) ,$
then for any $a\neq 0,$ 
\begin{equation}
aY+b\sim S_{\alpha }^{0}\left( \text{sign}\left( a\right) \beta ,\left\vert
a\right\vert \gamma ,a\delta +b\right)  \label{eq:INT:4}
\end{equation}

\item[(P2)] If $Y_{1},Y_{2},\ldots ,Y_{n}$ are pairwise independent and $%
Y_{j}\sim S_{\alpha }^{0}\left( \beta _{j},\gamma _{j},\delta _{j}\right) $
for $j=1,\ldots ,n$ then $\sum_{j=1}^{n}Y_{j}\sim S_{\alpha }^{0}\left(
\beta ,\gamma ,\delta \right) $ where%
\begin{equation}
\gamma ^{\alpha }=\sum_{j=1}^{n}\gamma _{j}^{\alpha },\quad \beta =\dfrac{%
\sum_{j=1}^{n}\beta _{j}\gamma _{j}^{\alpha }}{\sum_{j=1}^{n}\gamma
_{j}^{\alpha }}  \label{eq:INT:5}
\end{equation}

and%
\begin{equation}
\delta =\left\{ 
\begin{array}{lc}
\sum_{j=1}^{n}\delta _{j}+\tan \left( \pi \alpha /2\right) \left[ \beta
\gamma -\sum_{j=1}^{n}\beta _{j}\gamma _{j}\right] & \text{if }\alpha \neq 1
\\ 
\sum_{j=1}^{n}\delta _{j}+\dfrac{2}{\pi }\left[ \beta \gamma \ln \gamma
-\sum_{j=1}^{n}\beta _{j}\gamma _{j}\ln \gamma _{j}\right] & \text{if }%
\alpha =1%
\end{array}%
\right.  \label{eq:INT:6}
\end{equation}
\end{description}

Properties (P1) and (P2)\ for $n=2$ were given in \cite{ci:N98}. The
extension of Property (P2) to general $n$ is a straightforward induction.

\subsection{Temporal Aggregation}

The temporal aggregation of the stochastic process $\left\{ X_{t}\right\} $
is generally defined as the weighted sum of past and current process values.%
In this paper, we consider only a special case of temporal aggregation,
sometimes referred to as flow aggregation, where all the weights equal $1.$
The flow aggregation of $\left\{ X_{t}\right\} $ is given by%
\begin{equation}
S_{t}^{\left( r\right) }=\sum_{i=0}^{r-1}X_{t-i}.  \label{eq:INT:8}
\end{equation}%

Henceforth, we refer to $\left\{ S_{t}^{\left( r\right) }\right\} $ as the
temporal aggregation of $\left\{ X_{t}\right\} $ or the aggregated process,
to $r$ as the aggregation level and to $\left\{ X_{t}\right\} $ as the base
process.

We note that a moving average process is the temporal aggregation of an iid
process and that the temporal aggregation of a moving average process is
also a moving average process. A recent survey on temporal aggregation
can be found in \cite{ci:SV08}.

\subsection{Log Quantile Difference}

Let $\xi _{p}$ denote the pth quantile of some distribution function. At
quantile levels $p_{1},p_{2},$ such that $0<p_{1}<p_{2}<1$, we define the
log quantile difference $\zeta _{p_{1},p_{2}}$ to be%
\begin{equation}
\zeta _{p_{1},p_{2}}=\ln \left( \xi _{p_{2}}-\xi _{p_{1}}\right) .
\label{eq:INT:9}
\end{equation}%
We assume for the remainder of this paper, that any random variable on which
a log quantile difference is calculated has a positive density at $\xi
_{p_{1}}$ and $\xi _{p_{2}}.$ This assumption implies uniqueness of the
quantiles $\xi _{p_{1}}$ and $\xi _{p_{2}}$ and that the log quantile
difference is finite. Let us recall that the stable distributions satisfy
this condition.

\section{Log Quantile Difference of the Temporal Aggregation of a Stable
Moving Average Process}

Let $\left\{ X_{t}\right\} $ be the moving average process of order $q$
defined in $\left( \ref{eq:INT:1}\right) $ with stable innovations $\left\{
e_{t}\right\} ,$ let $\left\{ S_{t}^{\left( r\right) }\right\} $ denote the
temporal aggregation of $\left\{ X_{t}\right\} $ at aggregation level $r$
and let $\zeta _{p_{1},p_{2}}^{\left( r\right) }$ and $\zeta
_{p_{1},p_{2}}^{\left( 0\right) }$ denote respectively the log quantile
difference of $\left\{ S_{t}^{\left( r\right) }\right\} $ and $\left\{
e_{t}\right\} $ at quantile levels $p_{1},p_{2}.$ In this section, we show
under certain conditions on $\left\{ X_{t}\right\} ,$ that%
\begin{equation}
\zeta _{p_{1},p_{2}}^{\left( r\right) }=\alpha ^{-1}\ln \left( r\left\vert
\sum_{i=0}^{q}\theta _{i}\right\vert ^{\alpha }+g_{\alpha }\left( \theta
\right) \right) +\zeta _{p_{1},p_{2}}^{\left( 0\right) }  \label{eq:LQD:1}
\end{equation}

where%
\begin{equation}
g_{\alpha }\left( \theta \right) =\left( \sum_{i=0}^{q-1}\left\vert
\sum_{j=0}^{i}\theta _{j}\right\vert ^{\alpha }-q\left\vert
\sum_{i=0}^{q}\theta _{i}\right\vert ^{\alpha }+\sum_{i=1}^{q}\left\vert
\sum_{j=i}^{q}\theta _{j}\right\vert ^{\alpha }\right)  \label{eq:LQD:2}
\end{equation}

We start with a general result which applies to all stable moving average
processes.

\begin{theorem}
\label{th:LQD:1}The distribution of the aggregated process $\left\{
S_{t}^{\left( r\right) }\right\} $ is given by%
\begin{equation}
S_{t}^{\left( r\right) }\sim S_{\alpha }^{0}\left( \beta ^{\left( r\right)
},\gamma ^{\left( r\right) },\delta ^{\left( r\right) }\right)
\label{eq:LQD:3}
\end{equation}%
where 
\begin{equation}
\gamma ^{\left( r\right) }=\left( \sum_{j=0}^{r+q-1}\left\vert
c_{j}\right\vert ^{\alpha }\right) ^{1/\alpha }\gamma ^{\left( 0\right)
},\quad \beta ^{\left( r\right) }=\dfrac{\sum_{j=0}^{r+q-1}\text{sign}\left(
c_{j}\right) \left\vert c_{j}\right\vert ^{\alpha }}{\sum_{j=0}^{r+q-1}\left%
\vert c_{j}\right\vert ^{\alpha }}\beta ^{\left( 0\right) },
\label{eq:LQD:4}
\end{equation}
if $\alpha \neq 1$%
\begin{eqnarray}
\delta ^{\left( r\right) } &=&\left( \sum_{j=0}^{r+q-1}c_{j}\right) \delta
^{\left( 0\right) }+     \label{eq:LQD:6a} \\
&&\tan \left( \pi \alpha /2\right) \left[ \beta ^{\left( r\right) }\gamma
^{\left( r\right) }-\beta ^{\left( 0\right) }\gamma ^{\left( 0\right)
}\left( \sum_{i=0}^{r+q-1}\text{sign}\left( c_{j}\right) \left\vert
c_{j}\right\vert \right) \right] \notag
\end{eqnarray}
if $\alpha = 1$%
\begin{eqnarray}
\delta ^{\left( r\right) } &=&\left( \sum_{j=0}^{r+q-1}c_{j}\right) \delta
^{\left( 0\right) }+    \label{eq:LQD:6b} \\
&&\dfrac{2}{\pi }\left[ \beta ^{\left( r\right) }\gamma ^{\left( r\right)
}\ln \gamma ^{\left( r\right) }-\beta ^{\left( 0\right) }\gamma ^{\left(
0\right) }\left( \sum_{i=0}^{r+q-1}\text{sign}\left( c_{j}\right) \left\vert
c_{j}\right\vert \ln \left( \left\vert c_{j}\right\vert \gamma ^{\left(
0\right) }\right) \right) \right] \notag
\end{eqnarray}

and%
\begin{equation}
c_{j}=\left\{ 
\begin{array}{ll}
\sum_{i=0}^{j}\theta _{i} & j=0,\ldots ,q-1 \\ 
\sum_{i=0}^{q}\theta _{i} & j=q,\ldots ,r-1 \\ 
\sum_{i=j-r+1}^{q}\theta _{i} & j=r,\ldots ,r+q-1%
\end{array}%
\right. .  \label{eq:LQD:7}
\end{equation}
\end{theorem}

\begin{proof}
From the definition of the aggregated process $\left\{ S_{t}^{\left( r\right)
}\right\} ,$ we have for $r\geq q$ that%
\begin{eqnarray}
S_{t}^{\left( r\right) } &=&\sum_{i=0}^{r-1}X_{t-i}  \notag \\
&=&\sum_{i=0}^{r-1}\sum_{j=0}^{q}\theta _{j}e_{t-i-j}  \notag \\
&=&\sum_{j=0}^{r+q-1}c_{j}e_{t-j}  \label{eq:LQD:8}
\end{eqnarray}%
where $c_{j}$ is given by $\left( \ref{eq:LQD:7}\right) .$ An application of
properties (P1)\ and (P2)\ proves the theorem.
\end{proof}

Whilst Theorem \ref{th:LQD:1} provides formulae for the stable distribution
parameters of the aggregated process, in general it is not possible to
derive from these a formula for the log quantile difference of the
aggregated process. However, we can derive such a formula for those
processes where $\beta ^{\left( r\right) }=\beta ^{\left( 0\right) }.$ To
achieve this, we make use of the following lemma, which shows how the log
quantile difference of a random variable is affected by linear
transformations.

\begin{lemma}
\label{lem:LQD:1}Suppose $X$ is a random variable and $Y=aX+b$ for some $a>0$
and $b\in \mathbb{R}.$ Let $\xi _{X;p}$ and $\xi _{Y;p}$ denote respectively the pth quantile of 
$X$ and $Y,$ then%
\begin{equation}
\xi _{Y;p}=a\xi _{X;p}+b.  \label{eq:LQD:9}
\end{equation}%
Let $\zeta _{X;p_{1},p_{2}}$ and $\zeta _{Y;p_{1},p_{2}}$ denote
respectively the log quantile difference of $X$ and $Y$ at quantile levels $%
p_{1},p_{2},$ then%
\begin{equation}
\zeta _{Y;p_{1},p_{2}}=\ln a+\zeta _{X;p_{1},p_{2}}.  \label{eq:LQD:10}
\end{equation}
\end{lemma}

\begin{proof}
By assumption we have%
\begin{eqnarray}
p &=&P\left( X\leq \xi _{X;p}\right)  \notag \\
&=&P\left( Y\leq a\xi _{X;p}+b\right)  \label{eq:LQD:11}
\end{eqnarray}%
which proves $\left( \ref{eq:LQD:9}\right) $ and $\left( \ref{eq:LQD:10}%
\right) $ follows immediately.
\end{proof}

We can now prove the formula for $\zeta _{p_{1},p_{2}}^{\left( r\right) }$
in $\left( \ref{eq:LQD:1}\right) $ under certain conditions on the base
process, $\left\{ X_{t}\right\} $.

\begin{theorem}
\label{th:LQD:2} If the base process $\left\{ X_{t}\right\} $ satisfies
either

\begin{description}
\item[(A1)] 
\begin{equation}
\beta ^{\left( 0\right) }=0  \label{eq:LQD:12}
\end{equation}
\end{description}

or

\begin{description}
\item[(A2)] 
\begin{equation}
\sum_{j=0}^{i}\theta _{j}\geq 0\text{\quad for }i=0,\ldots ,q-1\text{ and }%
\sum_{j=i}^{q}\theta _{j}\geq 0\text{\quad for }i=1,\ldots ,q
\label{eq:LQD:13}
\end{equation}
\end{description}

then for $r\geq q$ the log quantile difference $\zeta _{p_{1},p_{2}}^{\left(
r\right) }$ is given by the formula in $\left( \ref{eq:LQD:1}\right) $.
\end{theorem}

\begin{proof}
From Theorem \ref{th:LQD:1}, we have for $r\geq q$ that the aggregated
process, $\left\{ S_{t}^{\left( r\right) }\right\} ,$ has a stable
distribution given by 
\begin{equation}
S_{t}^{\left( r\right) }\sim S_{\alpha }^{0}\left( \beta ^{\left( r\right)
},\gamma ^{\left( r\right) },\delta ^{\left( r\right) }\right)
\label{eq:LQD:14}
\end{equation}%
where $\beta ^{\left( r\right) },\gamma ^{\left( r\right) }$ and $\delta
^{\left( r\right) }$ are as shown in $\left( \ref{eq:LQD:4}\right)$ and 
$\left( \ref{eq:LQD:6a}\right) $ or $\left( \ref{eq:LQD:6b}\right) $. If
(A2) is satisfied, then all the $c_{j}$ terms in $\left( \ref{eq:LQD:7}%
\right) $ are non-negative and so 
\begin{equation}
\text{sign}\left( c_{j}\right) \left\vert c_{j}\right\vert ^{\alpha
}=\left\vert c_{j}\right\vert ^{\alpha }\text{\quad for }j=0,\ldots ,r+q-1.
\label{eq:LQD:15}
\end{equation}%
Note that $\sum_{j=1}^{q}\theta _{j}\geq 0$ implies that $%
\sum_{j=0}^{q}\theta _{j}>0.$ Thus, if either (A1)\ or (A2)\ is satisfied,
then%
\begin{equation}
\beta ^{\left( r\right) }=\beta ^{\left( 0\right) }  \label{eq:LQD:16}
\end{equation}%
and $\left\{ S_{t}^{\left( r\right) }\right\} $ is a scale and location
transformation of the innovations $\left\{ e_{t}\right\} .$ Thus 
\begin{equation}
\dfrac{S_{t}^{\left( r\right) }-\delta ^{\left( r\right) }}{\gamma ^{\left(
r\right) }}\sim \dfrac{e_{t}-\delta ^{\left( 0\right) }}{\gamma ^{\left(
0\right) }}  \label{eq:LQD:17}
\end{equation}%
and so from Lemma \ref{lem:LQD:1}%
\begin{equation}
\zeta _{p_{1},p_{2}}^{\left( r\right) }=\ln \left( \gamma ^{\left( r\right)
}/\gamma ^{\left( 0\right) }\right) +\zeta _{p_{1},p_{2}}^{\left( 0\right) }.
\label{eq:LQD:18}
\end{equation}%
Substituting $\left( \ref{eq:LQD:4}\right) $ and $\left( \ref{eq:LQD:7}%
\right) $ proves the theorem.
\end{proof}

Although for our purposes the formula for $\zeta _{p_{1},p_{2}}^{\left(
r\right) }$ in $\left( \ref{eq:LQD:1}\right) $ is only valid for integer
values of $r\geq q,$ nonetheless it is a function of $r$ which is
well-defined for all real positive values of $r.$ Formally, we can take
partial derivatives of $\zeta _{p_{1},p_{2}}^{\left( r\right) }$ with
respect to $\ln r,$ to get for $r\geq q$ 
\begin{equation}
\dfrac{\partial }{\partial \ln r}\zeta _{p_{1},p_{2}}^{\left( r\right)
}=\alpha ^{-1}\dfrac{r\left\vert \sum_{i=0}^{q}\theta _{i}\right\vert
^{\alpha }}{r\left\vert \sum_{i=0}^{q}\theta _{i}\right\vert ^{\alpha
}+g_{\alpha }\left( \theta _{1},\ldots ,\theta _{q}\right) }
\label{eq:LQD:19}
\end{equation}%
and%
\begin{equation}
\dfrac{\partial ^{2}}{\left( \partial \ln r\right) ^{2}}\zeta
_{p_{1},p_{2}}^{\left( r\right) }=\alpha ^{-1}\dfrac{r\left\vert
\sum_{i=0}^{q}\theta _{i}\right\vert ^{\alpha }g_{\alpha }\left( \theta
_{1},\ldots ,\theta _{q}\right) }{\left( r\left\vert \sum_{i=0}^{q}\theta
_{i}\right\vert ^{\alpha }+g_{\alpha }\left( \theta _{1},\ldots ,\theta
_{q}\right) \right) ^{2}}.  \label{eq:LQD:20}
\end{equation}%
and draw conclusions on the shape of $\zeta _{p_{1},p_{2}}^{\left( r\right)
} $.

\begin{corollary}
\label{cor:LQD:1}Let $\left\{ X_{t}\right\} $ be a MA(q) process satisfying
the conditions of Theorem \ref{th:LQD:2}. Then 
\begin{equation}
\lim_{r\rightarrow \infty }\dfrac{\partial }{\partial \ln r}\zeta
_{p_{1},p_{2}}^{\left( r\right) }=\alpha ^{-1}.  \label{eq:LQD:21}
\end{equation}%
For $r\geq q$, 
\begin{equation}
\text{sign}\left( \dfrac{\partial ^{2}}{\left( \partial \ln r\right) ^{2}}%
\zeta _{p_{1},p_{2}}^{\left( r\right) }\right) =\text{sign}\left( g_{\alpha
}\left( \theta \right) \right)  \label{eq:LQD:22}
\end{equation}%
and therefore%
\begin{equation}
\begin{tabular}{l}
$\text{if }g_{\alpha }\left( \theta \right) >0\text{ then }\zeta
_{p_{1},p_{2}}^{\left( r\right) }\text{ is convex in }\ln r,$ \\ 
$\text{if }g_{\alpha }\left( \theta \right) =0\text{ then }\zeta
_{p_{1},p_{2}}^{\left( r\right) }\text{ is linear in }\ln r,$ \\ 
$\text{if }g_{\alpha }\left( \theta \right) <0\text{ then }\zeta
_{p_{1},p_{2}}^{\left( r\right) }\text{ is concave in }\ln r.$%
\end{tabular}
\label{eq:LQD:23}
\end{equation}
\end{corollary}

\begin{remark}
\label{rem:LQD:3}If $\beta ^{\left( 0\right) }\neq 0$ and any of the $c_{j}$
terms in $\left( \ref{eq:LQD:7}\right) $ are negative, then $\beta ^{\left(
r\right) }\neq \beta ^{\left( 0\right) }.$ In that case, $\left( \ref%
{eq:LQD:17}\right) $ and consequently $\left( \ref{eq:LQD:1}\right) $ do not
hold. In general, equality relations for the quantiles of the sums of random
variables in terms of the quantiles of the summands are difficult to
achieve. (\cite{ci:WG86}, \cite{ci:LD89})
\end{remark}

\begin{remark}
\label{rem:LQD:4}In the special case where $\left\{ X_{t}\right\} $ is iid,
we have%
\begin{eqnarray}
\gamma ^{\left( r\right) } &=&r^{1/\alpha }\gamma ^{\left( 0\right) },\quad
\beta ^{\left( r\right) }=\beta ^{\left( 0\right) },  \notag \\
\delta ^{\left( r\right) } &=&\left\{ 
\begin{array}{lc}
r\delta ^{\left( 0\right) }+\tan \left( \pi \alpha /2\right) \beta ^{\left(
0\right) }\gamma ^{\left( 0\right) }\left( r^{1/\alpha }-r\right)  & \text{%
if }\alpha \neq 1 \\ 
r\delta ^{\left( 0\right) }+\dfrac{2}{\pi }\beta ^{\left( 0\right) }\gamma
^{\left( 0\right) }r\ln r & \text{if }\alpha =1%
\end{array}%
\right.   \label{eq:LQD:25}
\end{eqnarray}
and the expression for $\zeta _{p_{1},p_{2}}^{\left( r\right) }$ in $\left( %
\ref{eq:LQD:1}\right) $ reduces to%
\begin{equation}
\zeta _{p_{1},p_{2}}^{\left( r\right) }=\alpha ^{-1}\ln r+\zeta
_{p_{1},p_{2}}^{\left( 0\right) }.  \label{eq:LQD:26}
\end{equation}%
Note that the expressions for $\delta ^{\left( r\right) }$ in $\left( \ref%
{eq:LQD:25}\right) $ are different from those derived in Section 2.2 of 
\cite{ci:CCZW08} which the author believes to be in error.
\end{remark}

\begin{remark}
\label{rem:LQD:5}The derivatives in $\left( \ref{eq:LQD:19}\right) $ and $%
\left( \ref{eq:LQD:20}\right) $ and therefore the results of Corollary \ref%
{cor:LQD:1} do not depend on $p_{1},p_{2}$ for all $r\geq q$ and all $\alpha
.$
\end{remark}

In the next section we examine Corollary \ref{cor:LQD:1} in more detail for
the special cases of invertible MA(1)\ and MA(2)\ processes.

\section{Invertible Stable MA(1) and MA(2) Processes}

An invertible MA(q) process is one where all roots of the polynomial%
\begin{equation}
1+\theta _{1}z+\cdots +\theta _{q}z^{q}=0  \label{eq:ISP:1}
\end{equation}%
lie outside the complex unit circle, $\left\vert z\right\vert >1$. The
region of $\mathbb{R}^{q}$ 
in which invertible parameters reside is referred to as the
invertibility region. The invertibility region of MA(1)\ processes is the
set 
\begin{equation}
\left\{ \theta _{1}:\left\vert \theta _{1}\right\vert <1\right\} .
\label{eq:ISP:2}
\end{equation}%
The invertibility region of MA(2)\ processes is the set%
\begin{equation}
\left\{ \left( \theta _{1},\theta _{2}\right) :\theta _{2}<1\text{ and }%
\theta _{1}+\theta _{2}>-1\text{ and }\theta _{1}-\theta _{2}<1\right\} .
\label{eq:ISP:3}
\end{equation}%
\bigskip Expressions for the invertibility region of higher order MA
processes can be found in \cite{ci:W56}. In this section we identify regions
of the invertibility region of MA(1)\ and MA(2)\ processes where $g_{\alpha
}\left( \theta \right) $ is either positive, zero or negative for various
values of $\alpha .$ To conduct this analysis we require the following lemma.

\begin{lemma}
\label{lem:ISP:1}%
\begin{equation}
\begin{tabular}{ll}
$\text{if }x,y>0\text{ and }0<\alpha <1$ & then $\left\vert x+y\right\vert
^{\alpha }<\left\vert x\right\vert ^{\alpha }+\left\vert y\right\vert
^{\alpha }$ \\ 
$\text{if }x,y>0\text{ and }\alpha =1$ & then $\left\vert x+y\right\vert
^{\alpha }=\left\vert x\right\vert ^{\alpha }+\left\vert y\right\vert
^{\alpha }$ \\ 
$\text{if }x,y>0\text{ and }1<\alpha \leq 2$ & then $\left\vert
x+y\right\vert ^{\alpha }>\left\vert x\right\vert ^{\alpha }+\left\vert
y\right\vert ^{\alpha }$%
\end{tabular}%
.  \label{eq:ISP:4}
\end{equation}
\end{lemma}

\begin{proof}
If a function $f$ is strictly convex on $\left( a,b\right) $, then from
Jensen's inequality for $a<x,y<b$ 
\begin{equation}
f\left( x\right) +f\left( y\right) <f\left( x+y\right)  \label{eq:ISP:5}
\end{equation}%
The relations in $\left( \ref{eq:ISP:4}\right) $ are proved by applying $%
\left( \ref{eq:ISP:5}\right) $ to the function $f\left( x\right) =-\left(
x^{\alpha }\right) $ for $0<\alpha <1,$ and to the function $f\left(
x\right) =x^{\alpha }$ for $1<\alpha \leq 2$.
\end{proof}

To assist this analysis we divide the invertibility region for an MA(2)
process into 5 sub-regions as shown in Figure \ref{fig:ISP:1}. These
sub-regions are defined as open sets, so that the entire invertibility
region consists of the union of the 5 sub-regions, the borders between them
and the origin. The inequalities defining these sub-regions are listed in $%
\left( \ref{eq:ISP:9}\right) .$

\begin{equation}
\begin{tabular}{ll}
Sub-region 1 = $\left\{ \theta :\theta _{1}<-1\text{ and }\theta _{2}<1\text{
and }\theta _{1}+\theta _{2}>-1\right\} $ \\ 
Sub-region 2 = $\left\{ \theta :\theta _{1}>-1\text{ and }\theta _{2}>0\text{
and }\theta _{1}+\theta _{2}<0\right\} $ \\ 
Sub-region 3 = $\left\{ \theta :\theta _{2}>0\text{ and }\theta _{2}<1\text{
and }\theta _{1}+\theta _{2}>0\text{ and }\theta _{1}-\theta _{2}<1\right\} $
\\ 
Sub-region 4 = $\left\{ \theta :\theta _{2}<0\text{ and } -1<\theta _{1}+\theta
_{2}<0\text{ and }\theta _{1}-\theta
_{2}<1\right\} $ \\ 
Sub-region 5 = $\left\{ \theta :\theta _{2}<0\text{ and }\theta _{1}+\theta
_{2}>0\text{ and }\theta _{1}-\theta _{2}<1\right\} $%
\end{tabular}
\label{eq:ISP:9}
\end{equation}

\begin{figure}[htb] \centering
\includegraphics[width=9cm]{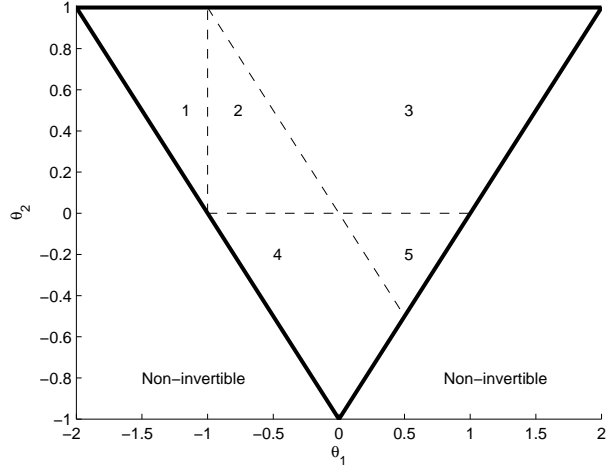}
\caption{The 5 sub-regions of the invertibility region of an MA(2) process.}
\label{fig:ISP:1}
\end{figure}

\begin{remark}
\label{rem:ISP:1}For an invertible MA(2)\ process, the set of values of $%
\left( \theta _{1},\theta _{2}\right) $ which satisfy condition (A2) in
Theorem \ref{th:LQD:2} consists of sub-region 3 and its borders with
sub-regions 2 and 5.
\end{remark}

The following theorem provides some properties of $g_{\alpha }\left( \theta
\right) $ for $\theta $ in sub-region 3.

\begin{theorem}
\label{th:ISP:2}If $\theta $ is an element of sub-region 3, then the
function $g_{\alpha }\left( \theta \right) $ satisfies the following
relations%
\begin{equation}
g_{\alpha }\left( \theta \right) \text{ is }\left\{ 
\begin{array}{cl}
>0 & \text{if }0<\alpha <1 \\ 
=0 & \text{if }\alpha =1 \\ 
<0 & \text{if }1<\alpha \leq 2%
\end{array}%
\right.  \label{eq:ISP:10}
\end{equation}
\end{theorem}

\begin{proof}
By definition for $q=2,$ we have,%
\begin{equation}
g_{\alpha }\left( \theta \right) =1+\left\vert 1+\theta _{1}\right\vert
^{\alpha }-2\left\vert 1+\theta _{1}+\theta _{2}\right\vert ^{\alpha
}+\left\vert \theta _{1}+\theta _{2}\right\vert ^{\alpha }+\left\vert \theta
_{2}\right\vert ^{\alpha }.  \label{eq:ISP:11}
\end{equation}

All $\theta $ in sub-region 3 satisfy $\theta _{1}+\theta _{2}>0,$ so using
Lemma \ref{lem:ISP:1} we get for $0<\alpha <1$%
\begin{equation}
\left\vert 1+\theta _{1}+\theta _{2}\right\vert ^{\alpha }<1+\left\vert
\theta _{1}+\theta _{2}\right\vert ^{\alpha }.  \label{eq:ISP:12}
\end{equation}%
All $\theta $ in sub-regions 3, satisfy $\theta _{1}>-1$ and $\theta _{2}>0,$
so using Lemma \ref{lem:ISP:1} we get for $0<\alpha <1$%
\begin{equation}
\left\vert 1+\theta _{1}+\theta _{2}\right\vert ^{\alpha }<\left\vert
1+\theta _{1}\right\vert ^{\alpha }+\left\vert \theta _{2}\right\vert
^{\alpha }.  \label{eq:ISP:13}
\end{equation}%
Therefore, for all $\theta $ in sub-region 3 and for $0<\alpha <1$ we have
that $g_{\alpha }\left( \theta \right) $ is the sum of two strictly positive
terms and so is strictly positive.

Similarly, for all $\theta $ in sub-region 3 and for $\alpha =1$ we have
that $g_{\alpha }\left( \theta \right) $ is the sum of two zero terms and so
is zero. Finally, for all $\theta $ in sub-region 3 and for $1<\alpha \leq 2$
we have that $g_{\alpha }\left( \theta \right) $ is the sum of two strictly
negative terms and so is strictly negative.
\end{proof}

Theorems similar to Theorem \ref{th:ISP:2} for the other sub-regions and the
borders between the sub-regions can be proven using the same approach. To
cover all the sub-regions and borders of the invertibility region of an
MA(2)\ process requires several such theorems. These are straightforward and
are omitted from this paper.

A\ sub-region is said to be positive, zero or negative for a given $\alpha $
if $g_{\alpha }\left( \theta \right) $ is respectively positive, zero or
negative for all points in the sub-region. A\ sub-region is said to be mixed
for a given $\alpha $ if there exist some points in the sub-region for which 
$g_{\alpha }\left( \theta \right) $ is positive and other points for which $%
g_{\alpha }\left( \theta \right) $ is negative. Similar descriptions are
used to describe the borders between the sub-regions. In Tables \ref%
{tab:ISP:1} and \ref{tab:ISP:2} we present a categorisation of all the
sub-regions and the borders between the sub-regions using these
descriptions. 

\begin{table}[h] \centering%
\begin{tabular}{|l|c|c|c|}
\hline
& \multicolumn{1}{|c|}{$0<\alpha <1$} & \multicolumn{1}{|c|}{$\alpha =1$} & $%
1<\alpha \leq 2$ \\ \hline
\multicolumn{1}{|l|}{Positive Sub-regions} & All & 1,2,4,5 & 1,4 \\ 
\multicolumn{1}{|l|}{Zero Sub-regions} & None & 3 & None \\ 
\multicolumn{1}{|l|}{Negative Sub-regions} & None & None & 3 \\ 
\multicolumn{1}{|l|}{Mixed Sub-regions} & None & None & 2,5 \\ \hline
\end{tabular}%
\caption{Categorisation of the sub-regions of the invertibility region of an MA(2) process into positive, zero, 
negative and mixed sub-regions with respect of g$\protect_\alpha(\theta)$.}%
\label{tab:ISP:1}%
\end{table}%

\begin{table}[h] \centering%
\begin{tabular}{|l|c|c|c|}
\hline
& $0<\alpha <1$ & $\alpha =1$ & $1<\alpha \leq 2$ \\ \hline
Positive Borders & All & $\left( 1,2\right) ,\left( 2,4\right) ,\left(
4,5\right) $ & $\left( 1,2\right) ,\left( 2,4\right) ,\left( 4,5\right) $ \\ 
Zero Borders & None & $\left( 2,3\right) ,\left( 3,5\right) $ & None \\ 
Negative Borders & None & None & $\left( 2,3\right) ,\left( 3,5\right) $ \\ 
\hline
\end{tabular}%
\caption{Categorisation of the borders between the sub-regions of the invertibility region of an MA(2) process into
 positive, zero and negative borders with respect of g$\protect_\alpha(\theta)$. 
We use (a,b) to denote the border between sub-regions a and b.}\label%
{tab:ISP:2}%
\end{table}%

The set of invertible MA(1) processes is equivalent to the borders
sub-regions 2 and 4 and between sub-regions 3 and 5.
For iid processes $g_{\alpha }\left(\theta\right)=0$ for all
$\alpha$. It is perhaps helpful to see the results of Tables \ref{tab:ISP:1} and \ref%
{tab:ISP:2} in graphical form as provided in Figure \ref{fig:ISP:2}.

\begin{figure}[h] \centering
\begin{tabular}{cc}
\includegraphics[width=5cm]{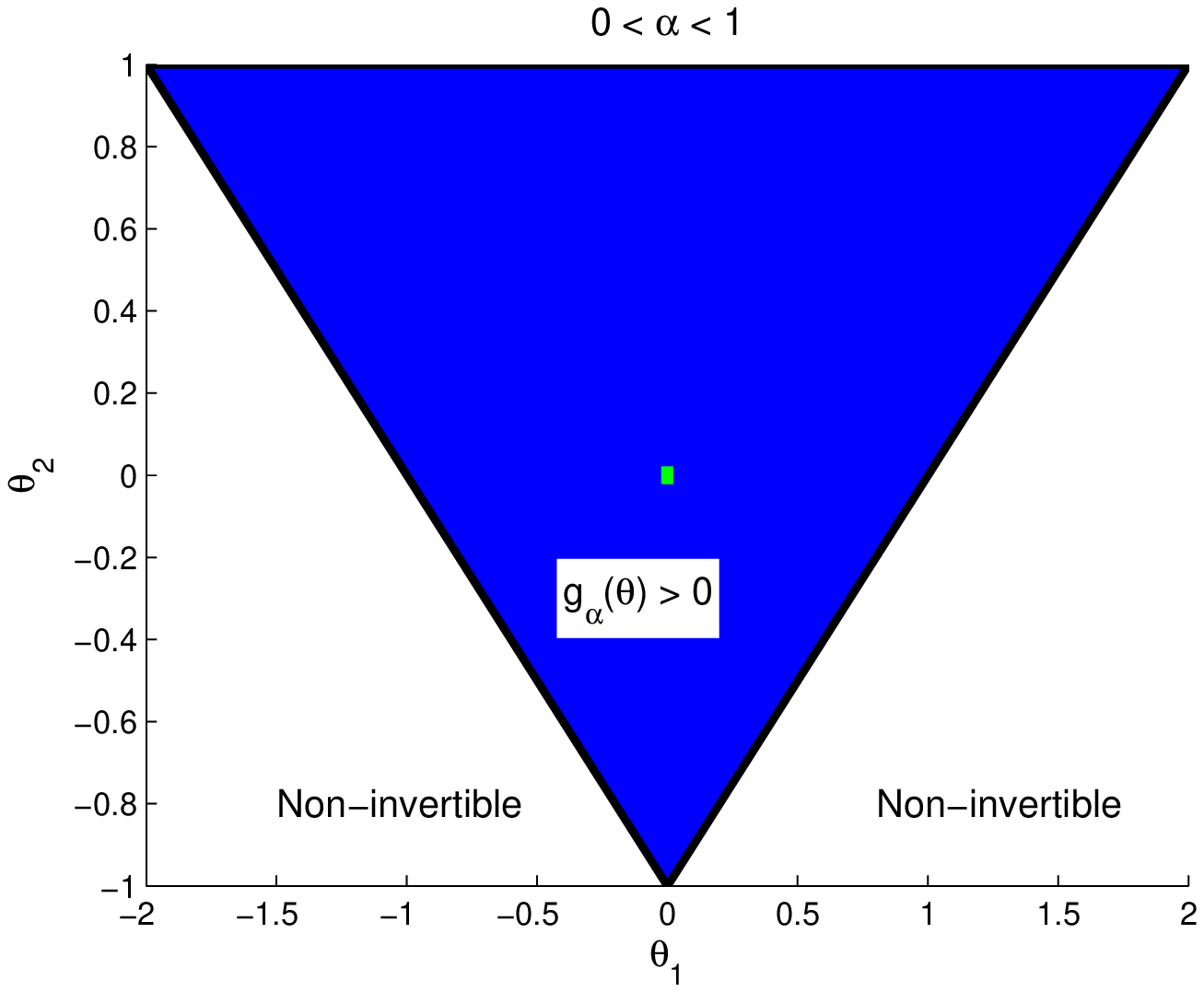} &
\includegraphics[width=5cm]{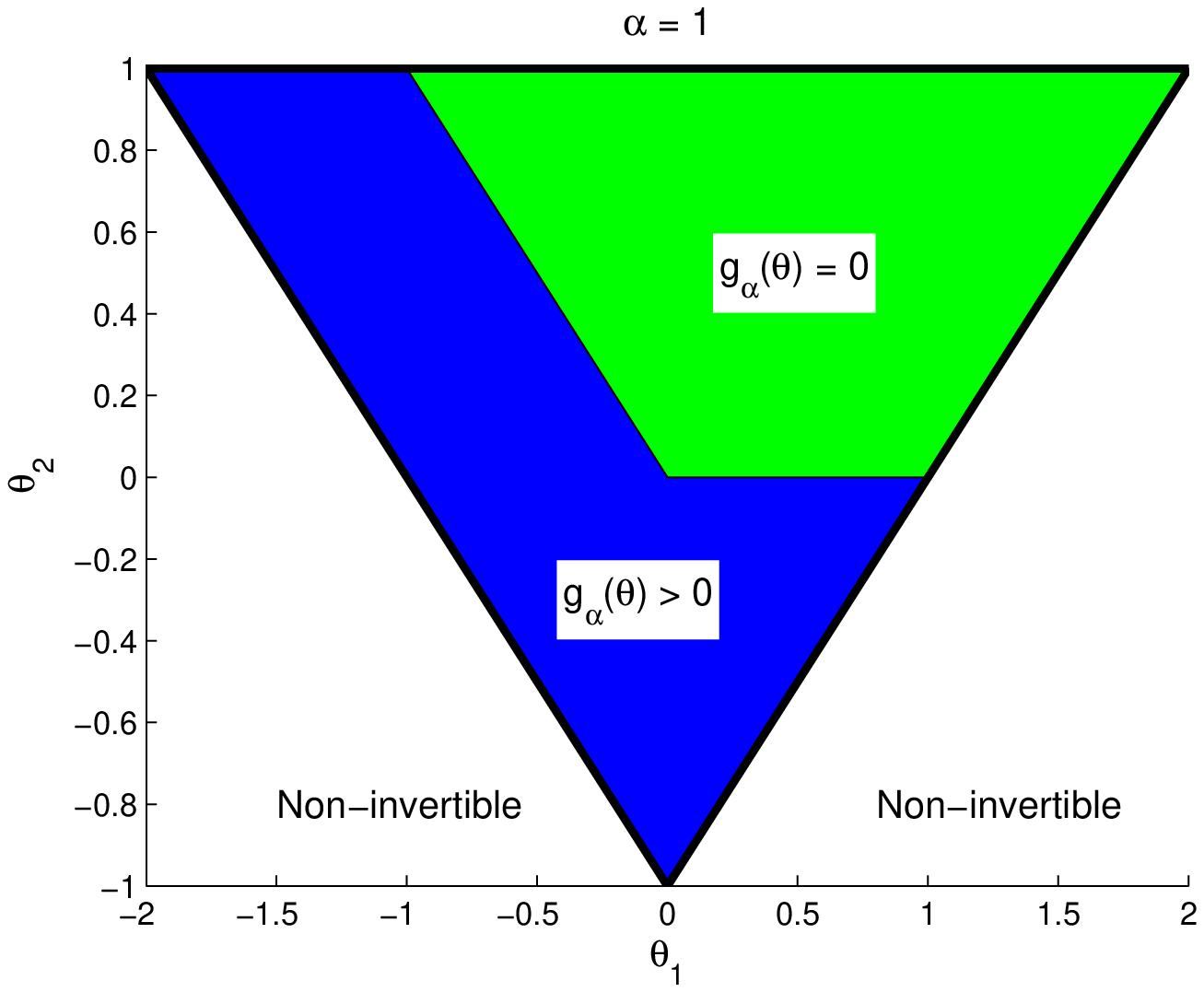} \\
\footnotesize{(a)}&
\footnotesize{(b)}\\
\includegraphics[width=5cm]{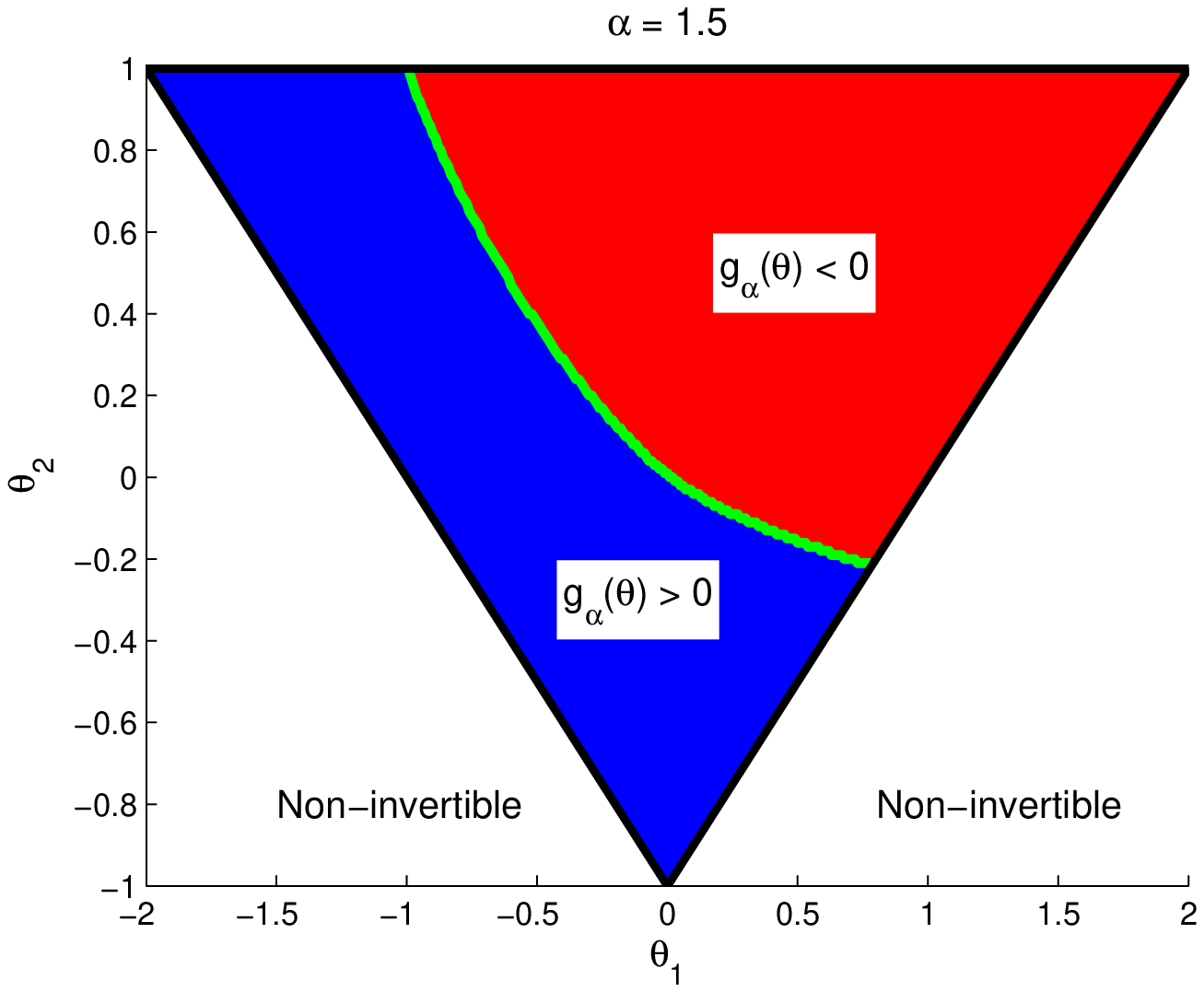} &
\includegraphics[width=5cm]{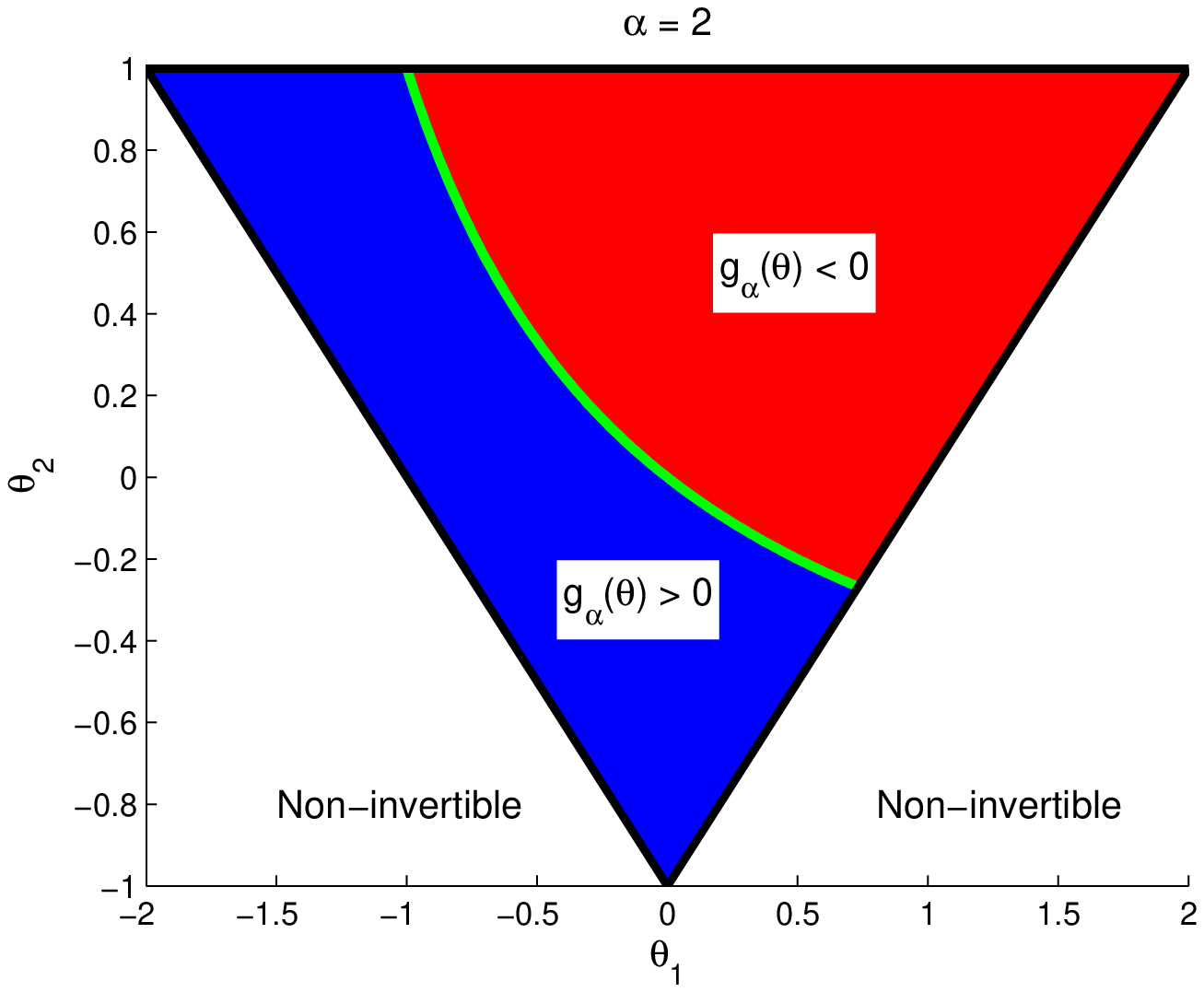} \\
\footnotesize{(c)}&
\footnotesize{(d)}\\
\end{tabular}
\caption{A graphical display of the categorisation of the invertibility region of MA(2) processes into 
positive (blue), zero (green) and negative (red) sub-regions for (a) $\protect 0 < \alpha < 1.0$, (b) $\protect\alpha$ = 1.0, 
(c) $\protect\alpha$ = 1.5 and (d) $\protect\alpha$ = 2.0.}
\label{fig:ISP:2}
\end{figure}

Figure \ref{fig:ISP:2}(a) is applicable to $g_{\alpha }\left( \theta \right) $ for
any $\alpha \in \left( 0,1\right) .$ Whilst Figures \ref{fig:ISP:2}(c)\ and %
\ref{fig:ISP:2}(d) appear similar, the locations of the respective green
lines, i.e. the sets 
\begin{equation}
\mathcal{D}_{\alpha }=\left\{ \theta :g_{\alpha }\left( \theta \right)
=0\right\} ,  \label{eq:ISP:15}
\end{equation}%
are not the same.

\begin{remark}
\label{rem:ISP:2}For an MA(2)\ process, it is straightforward to show that%
\begin{equation}
\mathcal{D}_{2}=\left\{ \theta :\theta _{1}+2\theta _{2}+\theta _{1}\theta
_{2}=0\right\} .  \label{eq:ISP:16}
\end{equation}%
For $1<\alpha <2,$ closed form expressions for $\mathcal{D}_{\alpha }$ have
not been obtained except to note that $\mathcal{D}_{\alpha }$ contains the
points $\theta =\left( 1,0,0\right) ^{\prime }$ and $\theta =\left(
1,-1,1\right) ^{\prime }.$ Strictly $\theta =\left( 1,-1,1\right) ^{\prime }$
is on the border of, but not in the invertibility region.
\end{remark}

To illustrate the behaviour of $\zeta _{p_{1},p_{2}}^{\left( r\right) }$
where $\theta $ lie in different sub-regions of the invertibility region, we
present plots of $\zeta _{0.50,0.95}^{\left( r\right) }$ for various
combinations of $\theta _{1},\theta _{2}$ and $\alpha $ in Figure \ref%
{fig:ISP:3}. 

\begin{figure}[h] \centering
\begin{tabular}{cc}
\includegraphics[width=5cm]{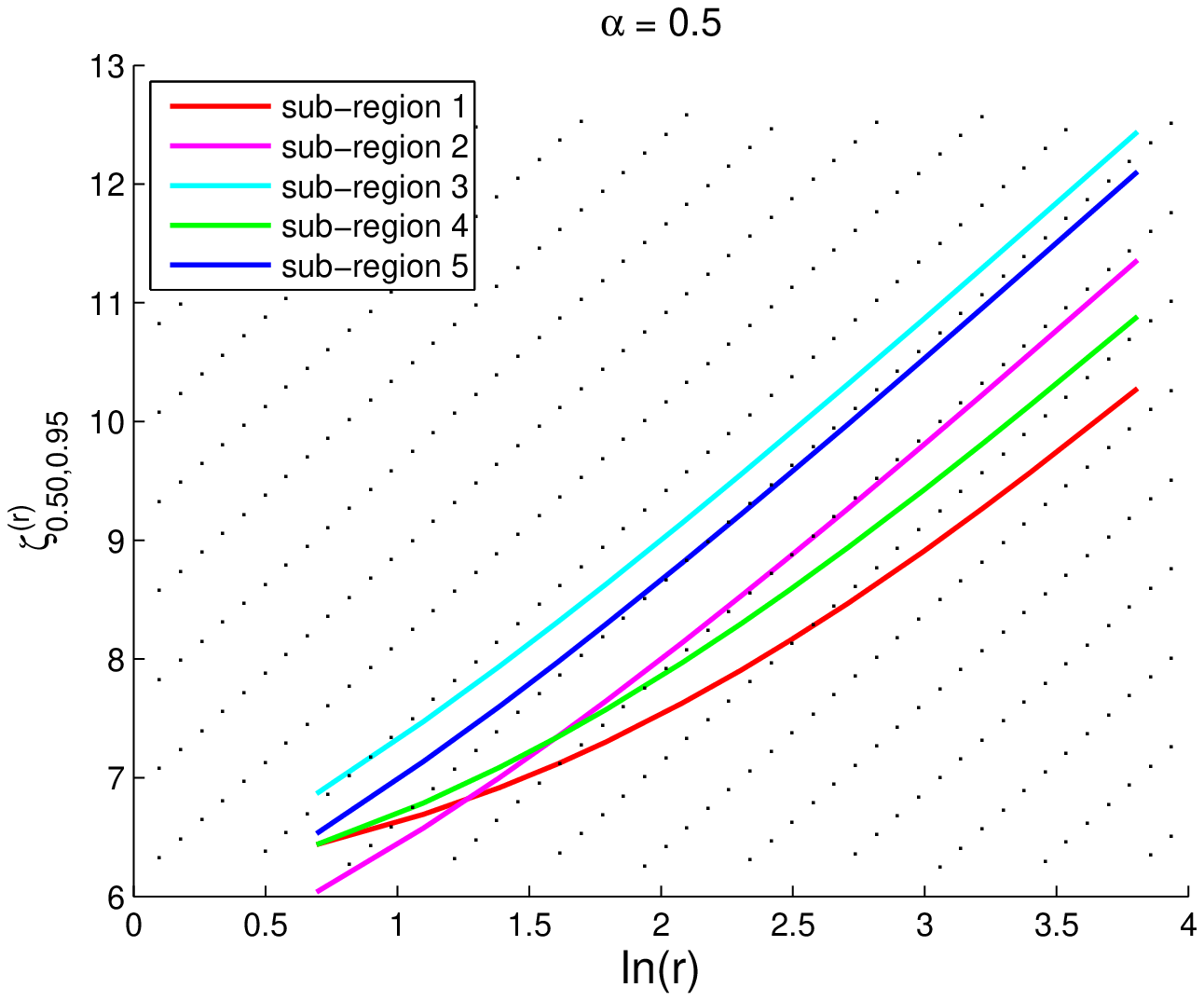} &
\includegraphics[width=5cm]{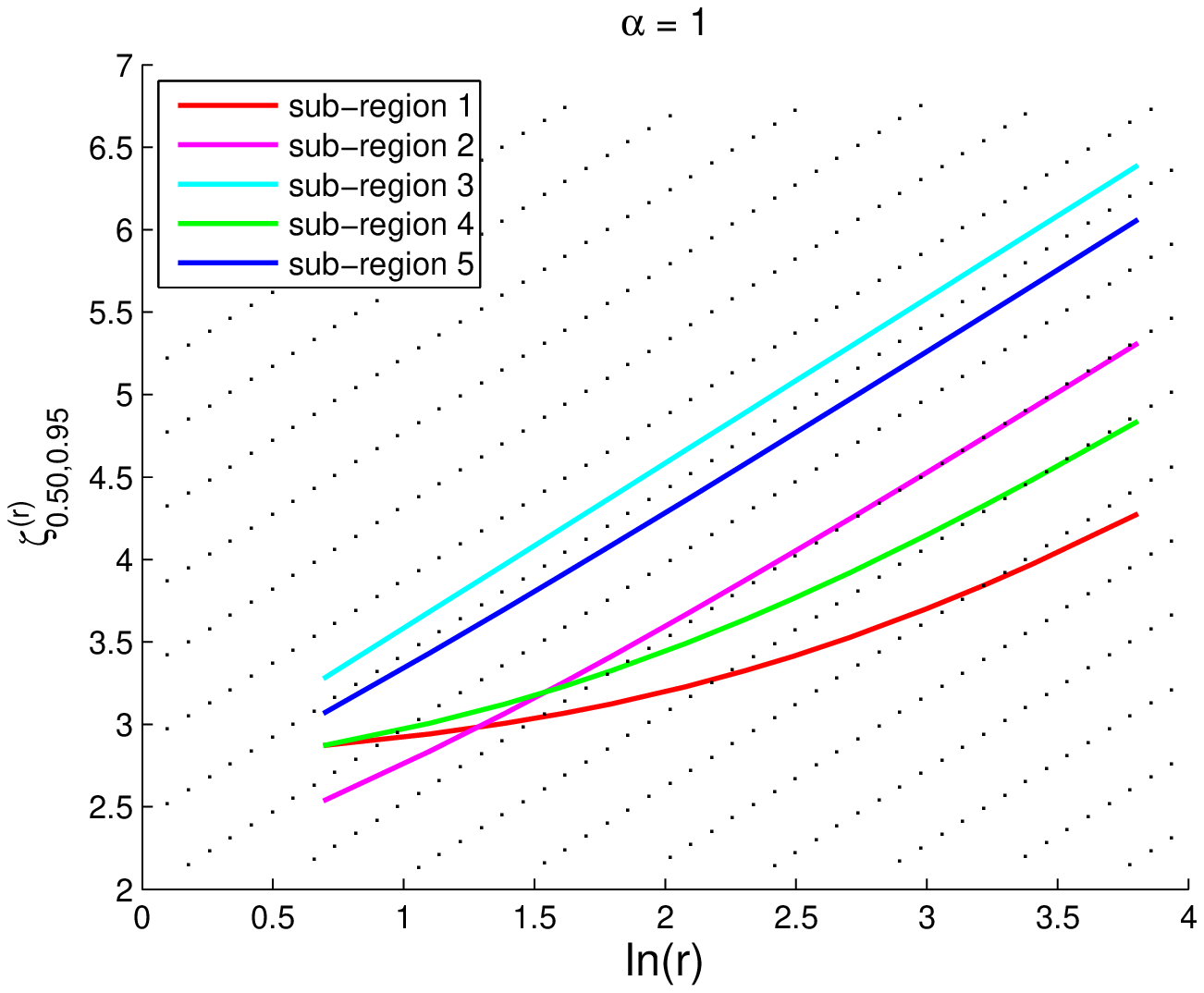} \\
\footnotesize{(a)}&
\footnotesize{(b)}\\
\includegraphics[width=5cm]{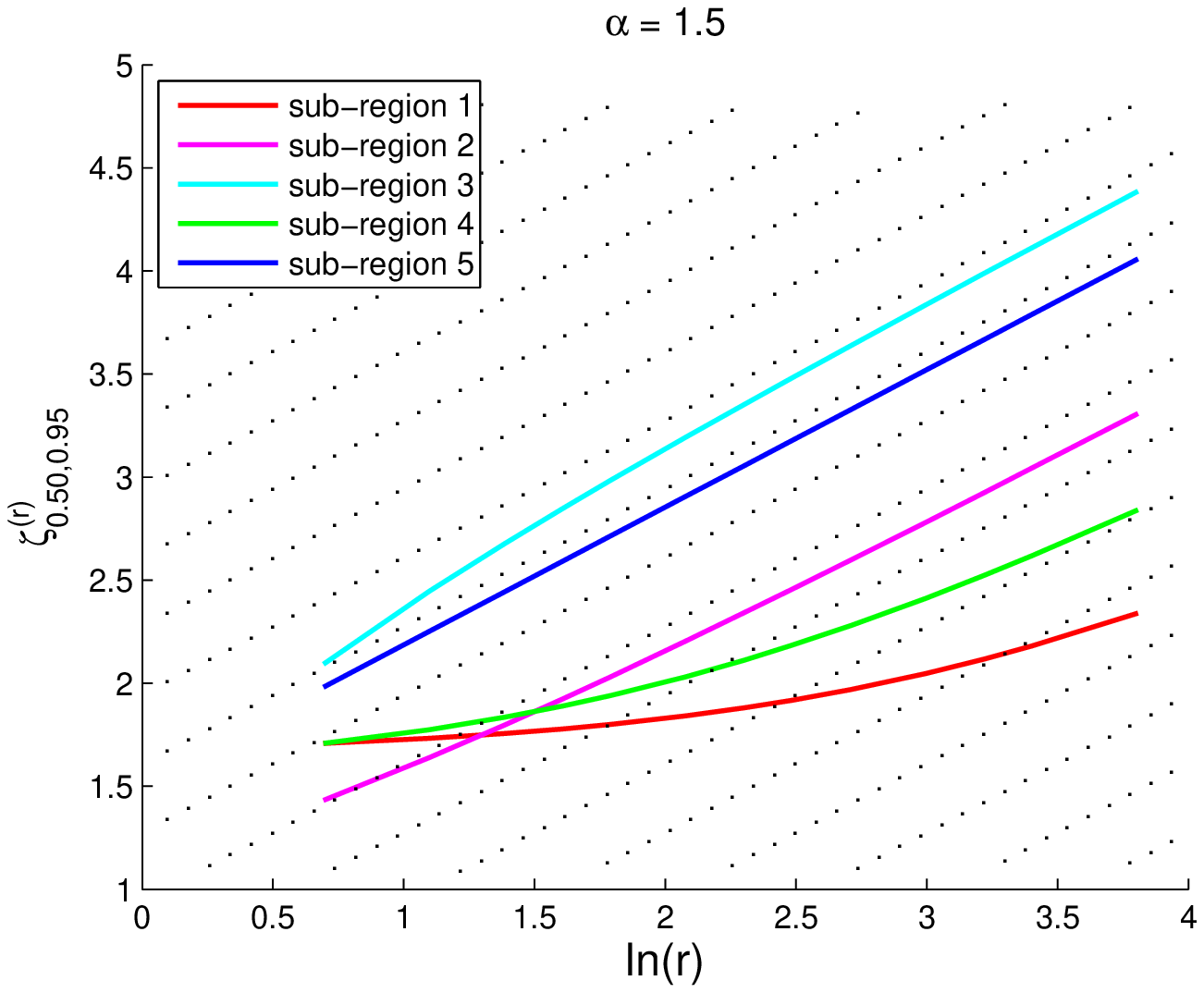} &
\includegraphics[width=5cm]{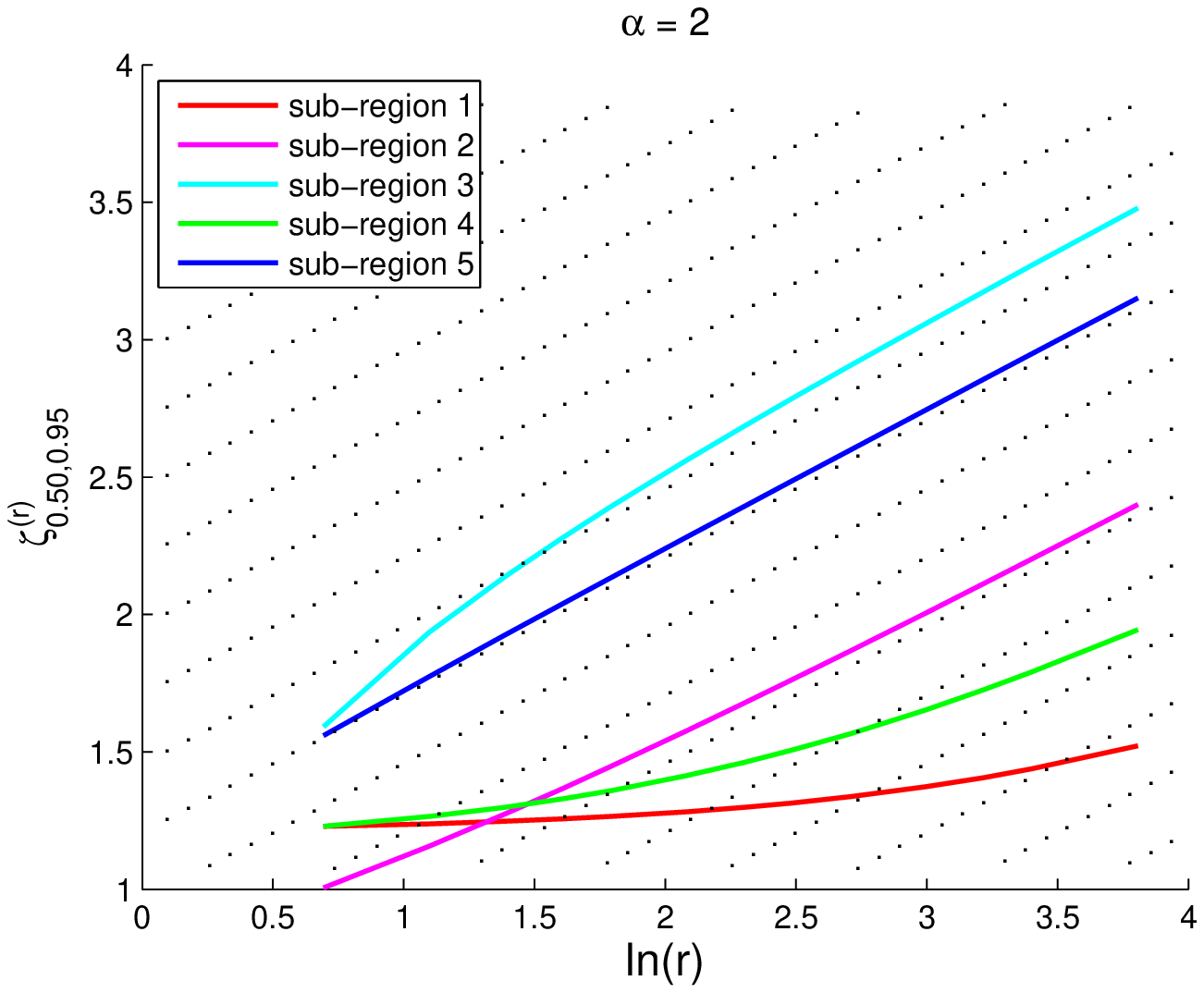} \\
\footnotesize{(c)}&
\footnotesize{(d)}\\
\end{tabular}
\caption{Plots of $\protect\zeta_{0.50,0.95}^{(r)}$ against the logarithm of the aggregation level for various  
symmetric stable MA(2) processes satisfying the conditions on Theorem \ref{th:LQD:1}.}
\label{fig:ISP:3}
\end{figure}

For each sub-figure in Figure \ref{fig:ISP:3}, we choose for sub-region 1: 
($\protect\theta_{1}$,$\protect\theta_{2}$) = (-1.4,0.6), 
sub-region 2: ($\protect\theta_{1}$,$\protect\theta_{2}$) = (-0.5,0.2), 
sub-region 3: ($\protect\theta_{1}$,$\protect\theta_{2}$) = (0.2,0.9), 
sub-region 4: ($\protect\theta_{1}$,$\protect\theta_{2}$) = (-0.2,-0.4) and
sub-region 5: ($\protect\theta_{1}$,$\protect\theta_{2}$) = (0.7,-0.2). 
The dotted parallel lines in Figure \ref{fig:ISP:3} have a slope $1/\alpha$ .

As shown in Corollary \ref{cor:LQD:1}, for each choice of $\alpha ,\theta$
in Figure \ref{fig:ISP:3}, it can be seen that the plot of 
$\zeta _{0.50,0.95}^{\left( r\right) }$ against $\ln r$ is concave, linear
or convex wherever $g_{\alpha }\left( \theta \right) $ is negative, zero or
positive and that the sign of $g_{\alpha }\left( \theta \right) $ agrees
with the results in Table \ref{tab:ISP:1}. In all cases the derivative $%
\partial \zeta _{p_{1},p_{2}}^{\left( r\right) }/\partial \ln r$ approaches $%
1/\alpha $ with increasing $r.$ The convergence of the derivative $\partial
\zeta _{p_{1},p_{2}}^{\left( r\right) }/\partial \ln r$ to $1/\alpha $ can
be much slower in the positive sub-regions than in the negative sub-regions.
The example shown in Figure \ref{fig:ISP:3}(d) for $\alpha =2$ and
sub-region 1, still has a derivative $\partial \zeta _{p_{1},p_{2}}^{\left(
r\right) }/\partial \ln r$ much less than $1/\alpha $ at an aggregation
level of $\exp \left( 3.8\right) \approx 45$.

\section*{Acknowledgement}
The author would like to thank Dr N. Kordzakhia and A/Prof A. Kozek for
their comments on this paper.

\bibliographystyle{model2-names}
\bibliography{acompat,OnTheLogQuantileDifferenceOfTheTemporalAggregationOfAStableMovingAverageProcess}

\newif\ifabfull\abfulltrue
\begin{thebibliography}{6}
\expandafter\ifx\csname natexlab\endcsname\relax\def\natexlab#1{#1}\fi
\providecommand{\url}[1]{\texttt{#1}}
\providecommand{\href}[2]{#2}
\providecommand{\path}[1]{#1}
\providecommand{\DOIprefix}{doi:}
\providecommand{\ArXivprefix}{arXiv:}
\providecommand{\URLprefix}{URL: }
\providecommand{\Pubmedprefix}{pmid:}
\providecommand{\doi}[1]{\href{http://dx.doi.org/#1}{\path{#1}}}
\providecommand{\Pubmed}[1]{\href{pmid:#1}{\path{#1}}}
\providecommand{\bibinfo}[2]{#2}
\ifx\xfnm\relax \def\xfnm[#1]{\unskip,\space#1}\fi
%Type = Article
\bibitem[{Chan et~al.(2008)Chan, Cheung, Zhang and Wu}]{ci:CCZW08}
\bibinfo{author}{Chan, W.}, \bibinfo{author}{Cheung, S.},
  \bibinfo{author}{Zhang, L.}, \bibinfo{author}{Wu, K.}, \bibinfo{year}{2008}.
\newblock \bibinfo{title}{Temporal aggregation of equity return time series
  models}.
\newblock \bibinfo{journal}{Mathematics and Computers in Simulation}
  \bibinfo{volume}{78}, \bibinfo{pages}{172--180}.
%Type = Article
\bibitem[{Liu and David(1989)}]{ci:LD89}
\bibinfo{author}{Liu, J.}, \bibinfo{author}{David, H.}, \bibinfo{year}{1989}.
\newblock \bibinfo{title}{Quantiles of sums and expected values of ordered
  sums}.
\newblock \bibinfo{journal}{Australian Journal of Statistics}
  \bibinfo{volume}{31}, \bibinfo{pages}{469--474}.
%Type = Article
\bibitem[{Nolan(1998)}]{ci:N98}
\bibinfo{author}{Nolan, J.}, \bibinfo{year}{1998}.
\newblock \bibinfo{title}{Parameterizations and modes of stable distributions}.
\newblock \bibinfo{journal}{Statistics and Probability Letters}
  \bibinfo{volume}{38}, \bibinfo{pages}{187--195}.
%Type = Article
\bibitem[{Silvestrini and Veredas(2008)}]{ci:SV08}
\bibinfo{author}{Silvestrini, A.}, \bibinfo{author}{Veredas, D.},
  \bibinfo{year}{2008}.
\newblock \bibinfo{title}{Temporal aggregation of univariate and multivariate
  time series models: A survey}.
\newblock \bibinfo{journal}{Journal of Economic Surveys} \bibinfo{volume}{22},
  \bibinfo{pages}{458--497}.
%Type = Article
\bibitem[{Watson and Gordon(1986)}]{ci:WG86}
\bibinfo{author}{Watson, R.}, \bibinfo{author}{Gordon, I.},
  \bibinfo{year}{1986}.
\newblock \bibinfo{title}{On quantiles of sums}.
\newblock \bibinfo{journal}{Australian Journal of Statistics}
  \bibinfo{volume}{28}, \bibinfo{pages}{192--199}.
%Type = Article
\bibitem[{Wise(1956)}]{ci:W56}
\bibinfo{author}{Wise, J.}, \bibinfo{year}{1956}.
\newblock \bibinfo{title}{Stationarity conditions for stochastic processes of
  the autoregressive and moving-average type}.
\newblock \bibinfo{journal}{Biometrika} \bibinfo{volume}{43},
  \bibinfo{pages}{215--219}.

\end{thebibliography}

\end{document}